\documentclass[11pt]{article}
\usepackage[sort, square, comma, numbers]{natbib} 
\usepackage{amsmath} 
\usepackage{lmodern}
\usepackage[T1]{fontenc}
\usepackage[english]{babel}
\usepackage[utf8]{inputenc}
\usepackage[a4paper]{geometry}
\usepackage{amssymb}
\usepackage[hypcap]{caption}
\usepackage{amsmath, amsthm, amsfonts, amscd}
\usepackage[nottoc,numbib]{tocbibind}
\usepackage{mathrsfs}
\usepackage{float}
\usepackage{graphicx}
\usepackage[labelformat=empty]{caption}
\usepackage{microtype}
\usepackage{todonotes}
\usepackage{makeidx}
\usepackage{verbatim}
\usepackage{hyperref}
\hypersetup{
colorlinks=false,
linkbordercolor=red,
pdfborderstyle={/S/U/W 1}
}
\usepackage{mdwlist}
\usepackage{mathabx}
\usepackage{mathtools}
\usepackage{titlesec}
\usepackage{tikz-cd}

\newtheorem{thm}{\scshape{Theorem}}[section]
\newtheorem{cor}[thm]{\scshape{Corollary}}
\newtheorem{prop}[thm]{\scshape{Proposition}}
\newtheorem{lemma}[thm]{\scshape{Lemma}}

\newtheorem{conj}[thm]{\scshape{Conjecture}}
\newtheorem*{theorem*}{Theorem}

\newtheorem*{prop*}{Proposition}
\newtheorem*{cor*}{Corollary}

\theoremstyle{definition}
\newtheorem{dfn}[thm]{\scshape{Definition}}

\newtheorem{remark}[thm]{\scshape{Remark}}

\def\G{\Gamma}

\def\subset{\subseteq}

\def\={\cong}

\def\mb{\mathbf}
\def\mc{\mathcal}
\def\mbb{\mathbb}

\def\Is{{\it Is}}
\def\ker{{\it ker}}
\def\G_0{G/\Is(\gamma_{3}(G))}

{}

\title{Metabelian groups: full-rank presentations, randomness and Diophantine problems}
\author{Albert Garreta, Leire Legarreta, Alexei Miasnikov, Denis Ovchinnikov}
\date{}

\begin{document}

\maketitle

\begin{abstract}
We study  metabelian groups $G$ given by full rank finite presentations $\langle A \mid R\rangle_{\mc{M}}$  in the  variety $\mathcal{M}$ of metabelian groups. We prove that $G$ is a product of a free metabelian subgroup of rank $\max\{0, |A|-|R|\}$ and a  virtually abelian normal subgroup, and that if $|R|\leq |A|-2$ then the Diophantine problem of $G$ is undecidable, while it is decidable if $|R|\geq |A|$. We further prove that if $|R|\leq |A|-1$ then in any direct decomposition of $G$  all,  but one, factors are virtually abelian. Since  finite presentations have full rank asymptotically almost surely, finitely presented metabelian groups satisfy all the aforementioned properties asymptotically almost surely. 
\end{abstract}

\tableofcontents

\section{Introduction}

In this paper we study finitely generated metabelian groups $G$ given by full rank finite presentations $G = \langle a_1, \ldots, a_n \mid r_1, \ldots, r_m\rangle_{\mc{M}}$ in the  variety $\mathcal{M}$ of metabelian  groups, random metabelian groups in the few relators model, and the Diophantine problem in such groups. We prove that the  Diophantine problem  in   the group $G$ above is  undecidable  if $m \leq n-2$ and decidable if $m \geq n$  (the case $m = n-1$ is still open); if $m \leq n-1$ then in any direct decomposition of $G$ all, but one,   factors are virtually abelian;  and, finally,  $G$ has a rather nice structure, namely, $G$ is a product of two subgroups $G= HL$, where $H$ is a free metabelian group of rank $\max\{n-m,0\}$ and $L$ is  a virtually abelian normal subgroup of $G$. The class of metabelian groups admitting full rank presentations is rather large. Indeed, it turns out that for fixed $n$ and $m$ a finite presentation $\langle a_1, \ldots, a_n \mid r_1, \ldots, r_m\rangle_{\mc{M}}$ has full rank asymptotically almost surely. In particular, random metabelian  groups (in the few relators model) have full rank presentations asymptotically almost surely. Hence, they asymptotically almost surely satisfy all the properties mentioned above. 

\subsection{Metabelian groups with full rank presentations}

Let $A = \{a_1, \ldots,a_n\}$ be a finite alphabet,  $A^{-1} = \{a_1^{-1}, \ldots,a_n^{-1}\}$, $A^{\pm 1} = A \cup A^{-1}$, $(A^{\pm 1})^\ast$  the set of all (finite) words in $A^{\pm 1}$, and  $R = \{r_1, \ldots,r_m \}$  a finite subset of $(A^\pm)^\ast$. We fix this notation for the rest of the paper. 

A pair $(A,R)$ is called  a {\em finite  presentation}, we  denote it by $\langle A \mid R\rangle$ or $\langle a_1, \ldots,a_n \mid r_1, \ldots,r_m \rangle$. If $\mathcal V$ is a variety or a quasivariety of groups then a finite presentation $\langle A \mid R\rangle$ determines a group $G = F_{\mathcal V}(A)/\langle \langle R \rangle \rangle$, where  $F_{\mathcal V}(A)$ is a free group in $\mathcal V$ with basis $A$ and $\langle \langle R \rangle \rangle$ is the normal subgroup of $F_{\mathcal V}(A)$ generated by $R$. In this case we write $G = \langle A \mid R\rangle_{\mathcal V}$.    The {\em relation matrix} $M(A,R)$ of the presentation $\langle A \mid R\rangle$
is an  $m \times n$ integral matrix whose $(i,j)$-th entry is the sum of the exponents of the $a_j$'s  that occur in $r_i$. It was introduced by Magnus in \cite{Magnus} (see also \cite{LS}, Chapter II.3,   for its ties  to relation modules in groups). The number $d = |A| - |R|$, if non-negative, is called the {\em deficiency} of the presentation $\langle A \mid R\rangle$ (see \cite{LS}, Chapter II.2 for a short survey on groups with positive deficiency). The matrix $M(A,R)$ has {\em full rank} if its rank is equal to $\min\{|A|,|R|\}$, i. e., it is the maximum possible.

We showed in \cite{GMO} that if a finitely generated nilpotent group $G$ admits a full-rank presentation, then $G$ is either virtually free nilpotent (provided the   deficiency $d \geq 2$), or  virtually cyclic (if $d = 1$), or finite (if $d \leq 0$).  

Groups given in  the variety ${\mathcal M}$  of all metabelian groups  by full rank presentations also have a rather restricted structure, as witnessed by the following result.

\begin{thm}\label{t: main_structure_theorem}
	Let $G$ be a metabelian group given by a full-rank presentation  $G=\langle A\mid R\rangle_\mc{M}$. Then there exist two finitely generated subgroups $H$ and $K$ of $G$ such that:
	\begin{itemize}
	    \item  [1.] $H$ is a free metabelian group of rank  $\max(|A| -|R|,0)$,
	    \item [2.]  $K$ is a virtually abelian group with  $|R|$ generators, and its normal closure $L = K^G$ in $G$ is again virtually abelian;
	    \item [3.] $G =  \langle H, K \rangle = LH$. 
	  	\end{itemize}
	 Moreover, there is an algorithm that given a presentation $G=\langle A\mid R\rangle_\mc{M}$ finds a free basis for  the subgroup $H$ and a  generating set in $|R|$ generators for the subgroup $K$. 
\end{thm}

\medskip \noindent
The result above complements the Generalized Freiheissatz for ${\mathcal M}$.  In \cite{Romanovskii} Romanovskii  proved that if a  metabelian  group $G$ is given in the variety  ${\mathcal M}$ by a finite presentation $\langle A \mid R\rangle_{\mc{M}}$ of deficiency $d  \geq 1$,  then there is a subset of generators $A_0 \subseteq A$ with $|A_0| = d$ which freely generates a free metabelian  subgroup $H = \langle A_0\rangle$.   Theorem  \ref{t: main_structure_theorem} shows that if the presentation $\langle A \mid R\rangle_{\mc{M}}$ has full rank then  there is a free metabelian of rank $d$ subgroup $H$ of $G$ and,  in addition,  there are  virtually abelian subgroups $K$ and $L$ as described  in items 2) and 3) above, such that $G = HL$.  Two remarks are in order here. First, the subgroup $H$ in Theorem  \ref{t: main_structure_theorem} is not necessary equal to $\langle A_0\rangle$ for a suitable  $A_0 \subseteq A$ as in the Romanovskii's result. However, for a given full rank presentation $G =  \langle A \mid R\rangle_{\mc{M}}$  of $G$ one can find algorithmically another full rank presentation of $G$, which is in Smith normal form (see Section \ref{se:Freiheissatz}), such that the subgroup $H$ is, indeed, generated by a suitable $A_0 \subseteq A$ and $K$ is generated by $A \smallsetminus A_0$.  Second, even if the presentation $G =  \langle A \mid R\rangle_{\mc{M}}$ is in 
 Smith normal form, but it is not of full rank, then the subgroups $K$ and $L$  as in Theorem  \ref{t: main_structure_theorem}  may not necessarily exist
(see details in Section \ref{se:Freiheissatz}).

In another direction, we showed  in \cite{GMO}  that in any direct decomposition of a nilpotent group $G$ given in the nilpotent variety   $\mathcal{N}_c$, $c \geq 2$, by a finite full rank presentation of deficiency $\geq 1$,  all, but one, direct factors are finite.  

A similar result holds in the variety $\mathcal{M}$ as well.

\begin{thm}\label{t: Full-rank_direct_decomposition}
	Let $G$ be a finitely generated metabelian group given by a full-rank presentation $G=\langle A\mid R \rangle_{\mc{M}}$ such that $|R|\leq |A|-1$. Then  in any direct decomposition of $G$ all, but one, direct factors are virtually abelian. 
\end{thm}

\subsection{Diophantine problems}

In Section 3 we study the Diophantine problem in finitely generated metabelian  groups given by  full rank presentations. This  is a continuation of research  in \cite{GMO2} and \cite{GMO}.

  Recall, that the \emph{Diophantine problem} in an algebraic structure $\mathcal{A}$ (denoted $\mc{D}(\mc{A})$)  is the task to determine whether or not  a given finite system of equations with constants in $\mc{A}$ has a  solution in $\mathcal{A}$. $D(\mathcal{A})$ is \emph{decidable} if there is an algorithm that given a finite system $S$ of equations with constants in $\mathcal{A}$ decides whether or not $S$ has a solution in $\mc{A}$. Furthermore, $\mc{D}(\mc{A})$ is  \emph{reducible} to  $\mc{D}(\mc{M})$, for   another structure $\mc{M}$,  if there is an algorithm that for any finite system of equations $S$  in $\mc{A}$ computes a finite system of equations $S_{\mc{M}}$ in $\mc{M}$ such that $S$ has a solution in $\mc{A}$ if and only if $S_{\mc{M}}$  has a solution in $\mc{M}$. 

 Note that due to the classical result of Davis,  Putnam, Robinson and Matiyasevich, the Diophantine problem  $\mc{D}(\mathbb{Z})$ in the ring of integers $\mathbb{Z}$ is undecidable \cite{mat, DPR}.  Hence if $\mc{D}(\mathbb{Z})$ is reducible to $\mc{D}(\mc{M})$, then  $\mc{D}(\mc{M})$ is also  undecidable.  

To prove  that  $\mc{D}(\mc{A})$ reduces to  $\mc{D}(\mc{M})$ for some structures  $\mc{A}$ and  $\mc{M}$  it suffices to show that  $\mc{A}$ is  interpretable by equations (or \emph{e-interpretable}) in $\mc{M}$.  E-interpretability is a variation of the classical notion of the first-order interpretability, where instead of arbitrary first-order formulas finite systems of equations are used as the interpreting formulas (see Definition \ref{d: e-int} for details).  The main relevant property of such interpretations is that if $\mc{A}$ is e-interpretable in $\mc{M}$ then $\mc{D}(\mc{A})$  is reducible to $\mc{D}(\mc{M})$ by a polynomial time many-one reduction (Karp reductions).

\begin{thm}\label{t: Diophantine_Problem_full_rank}
	Let $G$ be a metabelian group given by a full-rank presentation $G=\langle A\mid R \rangle_{\mc{M}}$. Then the following hold:
	\begin{enumerate}
	    \item If $|R| \leq |A|-2$ then the ring of integers $\mbb{Z}$ is e-interpretable in $G$, and the Diophantine problem in $G$ is undecidable.
	    \item If $|R| \geq |A|$ then the Diophantine problem of $G$ is decidable (in fact, the first-order theory of $G$ is decidable).
	\end{enumerate}
\end{thm}

This result is analogous to the one obtained for nilpotent groups in \cite{GMO}.

\begin{remark}
For the case of deficiency $1$ (i.e., $m=n-1$), decidability of the Diophantine problem over $G$ remains an interesting open problem. 
The recent work \cite{KLM_19} proves decidability of the Diophantine problem in $BS(1,n)=\langle a_1,a_2\mid a_1^{n-1}=[a_1,a_2]\rangle$.
Also  some deficiency $1$  presentations define cyclic groups, which have decidable Diophantine problem  \cite{Ershov}. 
\end{remark}
\begin{conj}
Let $G$ be a metabelian group given by a full-rank presentation $G=\langle A\mid R \rangle_{\mc{M}}$. If $|A|-|R| = 1$ then 
the Diophantine problem in $G$  is decidable. 

\end{conj}


\subsection{Random finitely presented groups: the few-relations model}\label{s: few_relators_model}

The first notion of \emph{genericity} or a \emph{random group} in the class of finitely presented groups is due to Gromov \cite{gromov_hyperbolic}, where he introduced what is now known as the few relators  model. A slightly different  approach was suggested by Olshanskii \cite{olshansky_random}  and Arzhantseva and Olshanskii \cite{arzhantsevsa_olshanski}. 

Nowadays, the few relators model can be described as follows.   Let $m, n$ be fixed positive integers. Consider, in the notation above,  the set $S(n,m)$ of all finite presentations $\langle A \mid R\rangle$ with  $|A| = n, |R| = m$. For a given positive integer $\ell$ consider a finite subset $S(n,m,\ell)$ of $S(n,m)$ which consists of all presentations $\langle A \mid R\rangle \in S(n,m)$, where each relator in $R$ has length precisely $\ell$. Now for a given property of groups $P$ consider a subset $S_P(n,m,\ell)$ of $S(n,m,\ell)$  of all presentations $\langle A \mid R\rangle \in S(n,m,\ell)$ which define groups that  satisfy $P$. The property $P$ is termed \emph{$(n,m)$-generic} if 
$$
\lim_{\ell \to \infty}\frac{|S_P(n,m,\ell)|}{|S(n,m,\ell)|} = 1.
$$
Here $|X|$ denotes the cardinality of a set $X$.
In this event we also say sometimes that   $P$ holds for $\langle A \mid R\rangle$ \emph{asymptotically almost surely as $\ell \to \infty$}. 

In the book \cite{gromov_hyperbolic} Gromov stated that the group property of being hyperbolic is $(n,m)$-generic for all $n$ and $m$. Later Olshanskii \cite{olshansky_random} and Champetier \cite{champetier}  gave rigorous proofs of this result. We refer to \cite{Olivier} for a survey on random finitely presented groups and to Kapovich and Schupp \cite{kapovich_shupp} on group theoretic models of randomness and genericity.  

Observe that  the few relators model described above  concerns classical finitely presented groups. However, this approach  can be utilized as well for finitely presented groups in any fixed variety of groups $\mathcal{V}$, in particular, for finitely generated metabelian groups given  in  the variety ${\mathcal M}$ of all metabelian groups by finite presentations $\langle A \mid R\rangle_{\mathcal M}$.  Note that every finitely generated group in ${\mathcal M}$ has a finite presentation in ${\mathcal M}$.

For the variety of nilpotent groups ${\mathcal N}_c$  (of a fixed nilpotency class $c$), this approach has been studied recently in \cite{Duchin, GMO}. Other models of randomness for the groups in ${\mathcal N}_c$   can be found in \cite{Dymarz, GMO2}. To the best of our knowledge, there is no study of random metabelian groups (in any model) prior to this paper. 

The following result is fundamental to our approach.

\begin{thm} \cite{GMO} \label{t: random_walks}
Let $R$ be a set of $m$ words of length $\ell$ in an alphabet $A^{\pm 1} = \{a_1^{\pm 1}, \dots, a_n^{\pm 1} \}$, i.e.\ each word is obtained by successively concatenating  randomly chosen letters from $A^{\pm 1}$ with uniform probability. Then $M(A, R)$ has full rank (i.e.\ ${\it rank}(M(A, R))=\min\{n, m\}$) asymptotically almost surely as $\ell \to \infty$.
\end{thm}

\begin{cor}
Let $\langle A\mid R\rangle=\langle a_1, \dots, a_n \mid r_1,\dots, r_m\rangle$ be a  presentation where all relators $r_i$ have length $\ell> 0$. Then the presentation $\langle A\mid R\rangle$ has full rank asymptotically almost surely as $\ell$ tends to infinity.
\end{cor}

One of the main appeals of full-rank presentations is that they occur asymptotically almost surely in the few-relators model for random groups in any variety $\mathcal{V}$, in particular,  in the variety ${\mathcal M}$.
Hence, random metabelian  groups (in the few relators model) have full rank presentations asymptotically almost surely. Therefore, all the properties described above for metabelian groups given by full rank presentations are generic in the class of finitely presented  metabelian groups. We refer to Section \ref{se:random} for precise statements of the results.

\section{Structural properties of  metabelian groups given by full-rank presentations}

Throughout the paper we use the following notation.

Let $G$ be a group. By $\gamma_i(G)$ we denote the $i$-th term of the lower central series of $G$, that is $\gamma_1(G)=G$ and $\gamma_{i+1}(G)=[G, \gamma_{i}(G)]$ for all $i\geq 1$.  By $G^{(i)}$ we denote the $i$-th term of the derived series of $G$, that is $G^{(0)}=G$ and $G^{(i+1)}=[G^{(i)}, G^{(i)}]$.

We denote by  $\mc{N}_c$ and $\mc{M}$ the families of nilpotent groups of nilpotency class at most $c$,  and of metabelian groups, respectively. In general, we refer to books \cite{LS,Segal} for the standard facts and notation in group theory and to \cite{Neumann} for basic notions regarding varieties.

\subsection{A variation of the Generalized Freiheissatz} \label{se:Freiheissatz}

In this section we obtain some structural results on full rank metabelian groups, in particular, Theorem \ref{t: main_structure_theorem}.   

To prove this theorem we need some results on presentations in Smith normal form. 

Recall (see, for example, \cite{sims}), that an integer matrix $A = (a_{ij})$ is in \emph{Smith normal form} if there is some integer $r \geq 0$ such that the entries $d_i = a_{ii}, 1\leq i \leq r$ are positive, $A$ has no other nonzero entries, and $d_i$ divides $d_{i+1}$ for $1\leq i<r$.

\begin{dfn}
A finite presentation $\langle A\mid R\rangle$  is said to be in \emph{Smith normal form} if the relation matrix $M(A, R)$ is in Smith normal form.
\end{dfn}

\begin{prop}\label{p: smith_normal_form}
For any   finite  presentation $\langle A \mid R \rangle$ there exists a finite presentation in Smith normal form  $\langle A' \mid R' \rangle$, with $|A| = |A'|$, $|R| = |R'|$,  and $rank(M(A,R))=rank(M(A',R'))$,  such that for any variety  $\mc{V}$ the groups $G = \langle A \mid R \rangle_{\mc{V}}$ and $G' = \langle A' \mid R' \rangle_{\mc{V}}$ are isomorphic. Moreover, such a presentation $\langle A' \mid R' \rangle$ and an isomorphism $G \to G'$ can be found algorithmically. 
\end{prop}

\begin{proof}
The presentation $\langle A' \mid R' \rangle$ can be obtained by repeatedly applying Nielsen transformations on the tuples  $A$ and $R$. Indeed, note that such a Nielsen transformation has the effect in the relation matrix $M(A,R)$  of adding or subtracting two columns or two rows, respectively. It is known (see \cite{sims})  that one can find a finite sequence of elementary row and column operations that transform $M(A,R)$ into its Smith normal form. Performing the corresponding Nielsen transformation on $\langle A \mid R \rangle$ one gets the required presentation $\langle A' \mid R' \rangle$.  The details of this procedure can be found in \cite{GMO}. %
\end{proof}

The complexity of finding a presentation in Smith normal form and other algorithmic considerations regarding full-rank presentations and presentations in Smith normal form will be studied in  upcoming work.

\begin{lemma}\label{t: structure_Smith_normal_form}
Let $G$ be a metabelian group given by a full-rank presentation in Smith normal form: 
$$
G=\langle A \mid R\rangle_\mc{M} = \langle a_1,\dots, a_n \mid a_1^{\alpha_1}=c_1, \dots, a_{m}^{\alpha_m}=c_m \rangle_\mc{M},
$$
where $c_i\in [F(A),F(A)]$, $\alpha_i \in \mbb{Z}\smallsetminus\{0\}$ for all $i=1, \dots, m$, and $m \leq n$.  Then the following holds:
\begin{itemize}
\item [1.]  $K= \langle a_1, \ldots a_m\rangle$ is a virtually abelian group, and its normal closure $L = K^G$ in $G$ is again virtually abelian;
\item  [2.] $H = \langle a_{m+1}, \dots, a_n \rangle$ is a free metabelian group of rank $n-m$;
\item [3.] $G = \langle H,K \rangle = HL$.
\end{itemize}

\end{lemma}
\begin{proof}
We show first that  $K = \langle a_1, \ldots a_m\rangle$ is virtually abelian. 
Note that $K\cap G'$ is abelian. Set $N=\alpha_1\cdots\alpha_m$. Then for all $g\in K$ one has  $g^N\in K \cap G'$. Hence  $K/K \cap G'$ is finite, so $K$ is virtually abelian. Similarly, for $L = K^G$ the subgroup $L \cap G'$ is normal in $L$ and abelian. The quotient $L/L \cap G'$ is abelian, of period $N$, and finitely generated (since $L = K(L\cap G'))$, hence finite. It follows that $L$ is abelian-by-(finite abelian). Note that $L$ might not be finitely generated. 

Now we show that $\langle a_{m+1}, \dots, a_n \rangle$ is a free metabelian group of rank $n-m$. Assume that $n-m\geq 2$. By Romanovski's aforementioned result \cite{Romanovskii} (see the discussion after Theorem \ref{t: main_structure_theorem}), we know that there exists a subset $A_1\subseteq A$ with $|A_1| = |A| - |R|$ such that $\langle A_1 \rangle$ is free metabelian freely generated by $A_1$. We claim that $A_1 = \{a_{m+1}, \dots, a_n\}$. Indeed, otherwise there exists $a_i \in A_1$ such that  $a_i^t \in G'$ for some $t\in \mbb{Z}\smallsetminus\{0\}$. Note that $|A_1| \geq 2$, so there is $a_j \in A_1$ with $i \neq j$.  It follows then that  $[a_i^t,[a_i,a_j]]=1$, a contradiction with the fact that $A_1$ freely generates $\langle A_1 \rangle$ as a free metabelian group. 

If $n-m = 1$ then the map $a_n\to 1$ and $a_i\to 0$ for all $i=1,\dots, n-1$ gives rise to a  homomorphism $G\to \mathbb{Z}$. 
Hence, $a_n$ has infinite order in $G$,  and $\langle a_n \rangle$ is a free metabelian group of rank $1$. We note in passing, that in the case when $m=n-1$ the element of $A$ that generates an infinite cyclic group (i.e.,  a free metabelian group of rank $1$) is not necessarily unique, e.g.,  in $BS(1,n)\cong \langle a_1,a_2\mid a_1^{n-1}=[a_1,a_2]\rangle$, both $a_1$ and $a_2$ generate an infinite cyclic subgroup.
\end{proof}

\begin{proof}[Proof of Theorem \ref{t: main_structure_theorem}]
	Let $G$ be a metabelian group given by a full-rank presentation  $G=\langle A\mid R\rangle_\mc{M}$.  By Proposition \ref{p: smith_normal_form} one can find algorithmically another presentation $G  \langle A'\mid R'\rangle_\mc{M} = \langle a_1',\dots, a_n' \mid r_1', \ldots, r_m' \rangle_\mc{M}$ of $G$ which is in Smith normal form. 
	
	If $m \leq n$ then Lemma \ref{t: structure_Smith_normal_form} applied to the presentation $\langle A', \mid R'\rangle$ gives  subgroups $H$ and $K$ with the required properties. From these we obtain  the required subgroups in $G$ by inverting the isomorphism $\langle A\mid R\rangle =\langle A'\mid R\rangle$. 
	
	On the other hand, if $m > n$, then $G$ is a quotient of the full-rank metabelian group with zero deficiency $G_1= \langle a_1', \ldots,a_n' \mid r_1', \ldots,r_n'\rangle$ (since $M(A',R')$ is in Smith normal form, this is a full-rank presentation), and by the previous case it follows that $G_1$ is virtually abelian. Since $G$ is a quotient of $G_1$, $G$ is also virtually abelian.
\end{proof}

\begin{remark}
\label{rem:2} 
Note, that if the presentation in Lemma \ref{t: structure_Smith_normal_form} is in  Smith normal form, but not of full rank, then by the Generalized Freiheissatz \cite{Romanovskii} the free subgroup $H = \langle A_0\rangle$ of $G$ exists, but there might not exists corresponding  subgroups $K$ and $L$ for $H$.   Indeed, let
$$
    G = \langle a_1, a_2, a_3, a_4 \mid [a_1,a_3] = 1, [a_2,a_4] = 1\rangle.
$$
Then $a_3,a_4$ generate a free metabelian group of rank 2, but $G/H^G$ is free metabelian of rank 2, so there is no a virtually abelian subgroup $L$, such that $G = HL$.
\end{remark}

\begin{remark}
Finally, note that  Theorem \ref{t: main_structure_theorem} or Lemma \ref{t: structure_Smith_normal_form}  may  not reveal fully the structure   of $G$ and how it is related to the subgroups $H$ and $K$. 
For example, consider 
$$
G=\langle a_1,a_2,a_3,a_4\mid a_1^m=[a_1,a_3],a_2^k=[a_2,a_4]\rangle.
$$
Lemma \ref{t: structure_Smith_normal_form}  tells us that $\langle a_3,a_4\rangle$ is freely generated by $a_3,a_4$ and $\langle a_1,a_2\rangle$ is a virtually abelian group. On the other hand, it is easy to check  that $\langle a_1,a_3\rangle\cong BS(1,m+1)$, $\langle a_2,a_4\rangle\cong BS(1,k+1)$, and $G$ is  the free metabelian product of two Baumslag-Solitar groups, which is not clear directly from the decomposition $G = HL$.
\end{remark}

\subsection{Direct decompositions}

In this section we  prove our main result on direct decomposition of metabelian groups given by full rank presentations, namely Theorem \ref{t: Full-rank_direct_decomposition}.

\begin{proof}[Proof of Theorem \ref{t: Full-rank_direct_decomposition}]
We showed in \cite{GMO} that if  $H$ is  a finitely generated nilpotent group of class $c \geq 2$ given by a finite full rank presentation $H = \langle A \mid R \rangle_{\mc{N}_c}$ with $|R| \leq |A|-1$,  then  in any direct decomposition of $H$ all, but one, direct factors are finite. We will use this fact in our proof. Let now $G$ be a finitely generated metabelian group given by a full-rank presentation $G=\langle A\mid R \rangle_{\mc{M}}$ such that $|R|\leq |A|-1$. Assume $G=G_1\times \dots \times G_k$ for some $k\geq 2$ and some subgroups $G_i$, $i=1, \dots, k$. Let $\pi$ be the natural projection of $G$ onto $H = G/\gamma_3(G)$. The quotient $H$ admits the full-rank presentation $\langle A\mid R\rangle_{\mc{N}_2}$, so by the result mentioned above  all, but one, say $\pi(G_k)$, of the groups $\pi(G_1), \dots, \pi(G_k)$ are finite. Hence $\ker\pi \cap G_i$ has finite index in $G_i$ for all $i=1,\dots, k-1$. However  $\ker\pi$ is abelian since $\ker\pi = \gamma_3(G) \leq [G,G]$. 
\end{proof}

\section{Diophantine problem}

\subsection{Diophantine problem and e-interpretability}
\label{se:e-interpretability}

In this section we introduce the technique  of interpretability by systems of equations.  It is nothing else than the classical model-theoretic technique of interpretability (see \cite{Marker, Hodges}), restricted to  systems of equations (equivalently, positive existential formulas without disjunctions). In \cite{ GMO1, GMO3} we used this technique to  study the Diophantine problem in different classes of solvable groups and rings.

\medskip

In what follows we often  use non-cursive boldface letters to denote tuples of elements: e.g.\ $\mathbf{a}=(a_1, \dots, a_n)$. Furthermore, we always assume that equations may contain constants from  the algebraic structure in which they are considered.

\begin{dfn}\label{d: e-def} 
 A set $D \subset M^m$ is called \emph{definable by systems of equations} in $\mathcal{M}$, or \emph{e-definable} in $\mathcal{M}$, if there exists a finite system of equations, say $\Sigma_{D}(x_1,\ldots,x_m, y_1, \dots, y_k)$, in the language of $\mc{M}$  such that 
for any tuple $\mathbf{a}\in M^m$, one has that  $\mathbf{a} \in D$ if and only if  the system  $\Sigma_{D}(\mathbf{a}, \mb{y})$ on variables $\mathbf {y}$ has a solution in $\mathcal{M}$. In this case $\Sigma_D$ is said to \emph{e-define} $D$ in $\mc{M}$.
\end{dfn}
\begin{remark}
Observe that, in the notation above, if $D \subset M^m$ is e-definable then it is definable in $\mathcal{M}$ by the formula $\exists \mb{y}  \Sigma_{D}(\mathbf{x}, \mb{y})$. Such formulas are called \emph{positive primitive}, or pp-formulas. Hence, e-definable subsets are sometimes called pp-definable. On the other hand, in number theory  such sets  are usually referred to as Diophantine ones. And yet, in algebraic geometry they can be described as projections of algebraic sets.
\end{remark}

\begin{dfn}\label{d: e-int}
An algebraic structure $\mathcal{A}= \left(A; f, \dots, r, \dots, c, \dots \right)$  is called \emph{e-interpretable} in another algebraic structure $\mathcal{M}$ if there exists $n\in \mathbb{N}$, a subset $D \subseteq \mathcal{M}^n$ and an onto map  (called the \emph{interpreting} map)
$
\phi: D \twoheadrightarrow \mathcal{A},
$ 
such that:
\begin{enumerate}
\item $D$ is e-definable in $\mathcal{M}$. 
\item  For every function $f=f(x_1, \dots, x_n)$ in the language of $\mathcal{A}$, the preimage by $\phi$ of the graph of $f$, i.e.\ the set $\{(x_1, \dots, x_k, x_{k+1}) \mid \phi(x_{k+1}) = f(x_1, \dots, x_k)\}$, is e-definable in $\mathcal{M}$. 
\item For every relation $r$ in the language of $\mathcal{A}$, and also for  the equality relation $=$ in $\mathcal{A}$,  the preimage by $\phi$ of the graph of $r$ is e-definable in $\mathcal{M}$.
\end{enumerate}
\end{dfn}

The following is a fundamental property of e-interpretability. Intuitively it states that if $\mc{A}$ is e-interpretable in $\mc{M}$, then any system of equations in $\mc{A}$ can be `encoded' as a system of equations in $\mc{M}$.
\begin{lemma}
\label{RedLemma}
Let  $\mathcal{A}$ be e-interpretable in $\mathcal{M}$ with  an  interpreting map  $\phi: D \twoheadrightarrow \mathcal{A}$ (in the notation of the Definition \ref{d: e-int}). Then for every finite system of equations $S(\mathbf{x})$ in $\mathcal{A}$, there exists   a finite system of equations $S^*(\mathbf{y}, \mathbf{z})$ in $\mathcal{M}$, such that if $(\mathbf{b}, \mathbf{c})$ is a solution to $S^*(\mathbf{y}, \mathbf{z})$ in $\mathcal{M}$, then $\mathbf{b}\in D$ and $\phi(\mathbf{b})$ is a solution to $S(\mathbf{x})$ in $\mathcal{A}$. Moreover, any solution $\mathbf{a}$ to $S(\mathbf{x})$ in $\mathcal{A}$ arises in this way, i.e.\  $\mathbf{a} = \phi(\mathbf{b})$ for some solution $(\mathbf{b}, \mathbf{c})$ to $S^*(\mathbf{y}, \mathbf{z})$ in $\mathcal{M}$, for some $i=1,\dots, k$. Furthermore, there is a polynomial time algorithm that constructs the system $S^*(\mathbf{y}, \mathbf{z})$ when given a system $S(\mathbf{x})$. 
\end{lemma}
\begin{proof}
It suffices to follow step by step the proof of Theorem 5.3.2 from \cite{Hodges}, which states that an analogue of the above holds when $\mc{A}$ is interpretable by first order formulas in $\mc{M}$. One needs to replace all first order formulas by systems of equations.
\end{proof}
 Now we state two key consequences of Lemma \ref{RedLemma}. 
\begin{cor}\label{RedCor}
If $\mathcal{A}$ is e-interpretable in $\mathcal{M}$, then $\mc{D}(\mc{A})$ is reducible to $\mc{D}(\mc{M})$. Consequently, if $\mc{D}(\mathcal{A})$  is undecidable, then $\mc{D}(\mathcal{M})$ is undecidable as well.
\end{cor}
\begin{cor}\label{transitivity}
e-interpetability is a transitive relation, i.e.,  if $\mc{A}_1$ is e-intepretable in $\mc{A}_2$, and $\mc{A}_2$ is e-interpretable in $\mc{A}_3$, then $\mc{A}_1$ is e-interpretable in $\mc{A}_3$.
\end{cor}

The following is a key  property of e-interpretability that is used below. 
\begin{prop}[\cite{GMO1}]\label{extra_prop1}
Let $H$ be a normal subgroup of a group $G$. If  $H$ is e-definable in $G$ (as a set) then the natural map  $\pi: G \to G/H$ is an e-interpretation of $G/H$ in $G$.   Consequently, $\mc{D}(G/H)$ is reducible to $\mc{D}(G)$. 
\end{prop}

\subsection{The Diophantine problem in metabelian groups given by full rank presentations}

We next discuss the Diophantine problem in  metabelian groups admitting a full-rank presentation. We will need the following result regarding the same problem in nilpotent groups:
\begin{thm}[\cite{GMO}]\label{t: Dioph_full_rank_nilpotent}
Let $G$ be a finitely generated nonabelian nilpotent group admitting a full rank presentation of deficiency at least $2$ (i.e.\ there are at least two more generators than relations). Then the ring of integers $\mbb{Z}$ is e-interpretable in $G$, and, in particular, the Diophantine problem of $G$ is undecidable.
\end{thm}

Next we recall the definition of (finite) verbal width. This notion  is conveniently related to definability by equations, as we see in  Proposition \ref{p: interp_verbal_elliptic}. 

Let $w=w(x_1, \dots, x_m)$ be a word on an alphabet of variables and its inverses $\{x_1, \dots, x_m\}^{\pm 1}$. The \emph{$w$-verbal subgroup} of a group $G$ is defined as $w(G) = \langle w(g_1,\ldots,g_m)\mid g_i\in G \rangle$, and  $G$ is said to have \emph{finite $w$-width} if there exists an integer $n$, such that every  $g \in w(G)$ can be expressed as a product of at most $n$ elements of the form $w(g_1,\ldots,g_m)^{\pm 1}$
. Hence, $w(G)$ is e-definable  in $G$ through the equation
$
x=\prod_{i=1}^n w(y_{i1}, \dots, y_{im}) w(z_{i1}, \dots, z_{im})^{-1}
$
on variables $x$ and $\{y_{ij}, z_{ij} \mid 1\leq i\leq n, \ 1\leq j \leq m\}$. 

\begin{prop}\label{p: interp_verbal_elliptic}
Let $G$ be a group and let $H$ be a  normal verbal subgroup of $G$ with finite verbal width. Then the quotient $G/H$ is e-interpretable in $G$.
\end{prop}
\begin{proof} 
By the argument above the subgroup $H$ is e-definable in $G$. Now the result follows from Proposition \ref{extra_prop1}.
\end{proof}

\begin{proof}[Proof of Theorem \ref{t: Diophantine_Problem_full_rank}]
Let $G=\langle A\mid R \rangle_{\mc{M}}$ be a full rank presentation of a metabelian group $G$. To prove 1) note first that since $\gamma_3(G) \geq [G',G']$  the 2-nilpotent quotient  $G/\gamma_3(G)$ admits a  presentation $\langle A\mid R\rangle_{\mc{N}_2}$ in  the variety $\mc{N}_2$ of nilpotent groups of class $\leq 2$. In particular, the group $G/\gamma_3(G)$ has a  full-rank presentation of deficiency at least 2.  By Theorem \ref{t: Dioph_full_rank_nilpotent}, the ring $\mbb{Z}$ is e-interpretable in $G/\gamma_3(G)$.   It is known that in a finitely generated metabelian group any verbal subgroup has finite width \cite{Segal10}, in particular, the subgroup $\gamma_3(G)$ has finite width in $G$.  Hence by Proposition \ref{p: interp_verbal_elliptic}  the group $G/\gamma_3(G)$ is e-interpretable in $G$. 
It follows from transitivity of e-interpretations that $\mbb{Z}$ is e-interpretable in $G$, so the Diophantine problem in $G$ is undecidable. 

To prove the second statement of the theorem observe that if $|R| \geq |A|$ then  by Theorem \ref{t: main_structure_theorem} the group $G$ is virtually abelian. Hence the Diophantine problem in $G$ is decidable (see  \cite{Ershov}).  
\end{proof}

\section{Random metabelian groups}
\label{se:random}

In this section we study random metabelian groups in the   few-relators model.
More precisely, we consider group presentations 
\begin{equation}\label{e: main_presentation2}
G=\langle a_1, \dots, a_n \mid r_1, \dots, r_m\rangle_{\mathcal{M}} = \langle A \mid R \rangle_{\mc{M}}
\end{equation}
in the variety of metabelian groups $\mc{M}$, where the set of generators $A = \{a_1, \ldots, a_n\}$ is fixed, the number of relations $m$ is also fixed,   and $R$ is a set of $m$ words of length $\ell$ in the alphabet $A^{\pm 1}$ chosen randomly and uniformly, as explained in Section \ref{s: few_relators_model}. We then study the asymptotic properties of $G$ as $\ell$ tends to infinity. 

As we mentioned in the introduction the key observation here is that due to Theorem \ref{t: random_walks}  a finite presentation in the variety of all groups (hence, any variety) has full rank asymptotically almost surely.

\begin{thm}
 Let $n, m\in \mbb{N}$, and let $G$ be a finitely generated metabelian group given by a presentation $\langle A\mid R\rangle_{\mc{M}}=\langle a_1, \dots, a_n \mid r_1, \dots, r_m \rangle_{\mc{M}}$, where all relators $r_i$ have length $\ell$. Then  the following holds asymptotically almost surely  as $\ell \to \infty$: There exist two finitely generated subgroups $H$ and $K$ of $G$ such that:
	\begin{itemize}
	    \item  [1.] $H$ is a free metabelian group of rank  $\max(|A| -|R|,0)$,
	    \item [2.]  $K$ is a virtually abelian group with  $|R|$ generators, and its normal closure $L = K^G$ in $G$ is again virtually abelian;
	    \item [3.] $G =  \langle H, K \rangle = LH$. 
	  	\end{itemize}
	 Moreover, in this case there is an algorithm that given a presentation $G=\langle A\mid R\rangle_\mc{M}$ finds a free basis for  the subgroup $H$ and a  generating set in $|R|$ generators for the subgroup $K$. 
\end{thm}
\begin{proof}
It follows from Theorems \ref{t: random_walks}  and \ref{t: main_structure_theorem}.
\end{proof}

\begin{thm}
Let $n, m\in \mbb{N}$, and let $G$ be a finitely generated metabelian group given by a presentation $\langle A\mid R\rangle_{\mc{M}}=\langle a_1, \dots, a_n \mid r_1, \dots, r_m \rangle_{\mc{M}}$, where all relators $r_i$ have length $\ell$. Then  the following hold asymptotically almost surely  as $\ell \to \infty$: 
	\begin{enumerate}
	    \item If $|R| \leq |A|-2$ then the ring of integers $\mbb{Z}$ is interpretable in $G$ by systems of equations, and the Diophantine problem in $G$ is undecidable.
	    \item If $|R| \geq |A|$ then the Diophantine problem of $G$ is decidable (in fact, the first-order theory of $G$ is decidable).
	\end{enumerate}
\end{thm}
\begin{proof}
It follows from  Theorems \ref{t: random_walks}  and \ref{t: Diophantine_Problem_full_rank}.
\end{proof}

\begin{thm}
Let $n, m\in \mbb{N}$, and let $G$ be a finitely generated metabelian group given by a presentation $\langle A\mid R\rangle_{\mc{M}}=\langle a_1, \dots, a_n \mid r_1, \dots, r_m \rangle_{\mc{M}}$, where all relators $r_i$ have length $\ell$. Assume $n\geq m-1$. Then  the following holds asymptotically almost surely  as $\ell \to \infty$:    in any direct decomposition of $G$ all, but one, direct factors are virtually abelian. 
\end{thm}
\begin{proof}
It follows from  Theorems \ref{t: random_walks}  and \ref{t: Full-rank_direct_decomposition}.
\end{proof} 

\section{Acknowledgements}

This work was supported by the Mathematical Center in Akademgorodok.

Additionally, the first named author was supported by the ERC grant PCG-336983. The first and second named authors were supported by the Basque Government  grant IT974-16, and by the Ministry of Economy, Industry and Competitiveness of the Spanish Government Grant MTM2017-86802-P.

\bibliography{bib}

\end{document}